\documentclass[11pt, reqno]{amsart}
\usepackage{amscd,amsmath,amssymb,amsthm,yfonts,mathrsfs}
\global\binoppenalty=10000

\newtheorem{theorem}{Theorem}[section]

\newtheorem{prop}[theorem]{Proposition}
\newtheorem{cor}[theorem]{Corollary}

\theoremstyle{definition}
\newtheorem{defn}[theorem]{Definition}
\newtheorem{remark}[theorem]{Remark}

\numberwithin{equation}{section}

\def\beq{\begin{equation}}
\def\eeq{\end{equation}}
\def\beqs{\begin{equation*}}
\def\eeqs{\end{equation*}}

\newcommand{\hr}[1]{\left(#1\right)}                                                    
\newcommand{\ha}[1]{\left\langle#1\right\rangle}                                        
\newcommand{\hc}[1]{\left\{#1\right\}}                                                  
\newcommand{\br}[1]{\bigl(#1\bigr)}                                                     

\newcommand{\sums}[1]{\sum\limits_{{#1}}}                                               

\newcommand{\longra}{\longrightarrow}

\def\ad{\operatorname{ad}}
\def\bgt{\mathfrak b}
\def\coreg{\operatorname{coreg}}

\def\dgt{\mathfrak d}
\def\D{\mathcal D}
\def\End{\operatorname{End}}
\def\g{\mathfrak g}
\def\id{\operatorname{id}}
\def\lhu{\leftharpoonup}
\def\mf{\mathfrak}
\def\Mod{\operatorname{Mod}}
\def\rhu{\rightharpoonup}
\def\sslash{/\!/}

\begin{document}

\title[Dual pairs of quantum moment maps]{Dual pairs of quantum moment maps and doubles of Hopf algebras}
\author{Gus Schrader}
\author{Alexander Shapiro}
\address{Department of Mathematics, University of California, Berkeley, CA 94720, USA}

\begin{abstract}
For any finite-dimensional Hopf algebra $A$ there exists a natural associative algebra homomorphism $D(A) \to H(A)$ between its Drinfeld double $D(A)$ and its Heisenberg double $H(A)$. We construct this homomorphism using a pair of commuting quantum moment maps, and then use it to provide a homomorphism of certain reflection equation algebras. We also explain how the quantization of the Grothendieck-Springer resolution arises in this context.
\end{abstract}

\maketitle

\section{Introduction}

The Grothendieck-Springer simultaneous resolution of a complex simple Lie group $G$ plays a central role in the geometric representation theory. Recall the if $B\subset G$ is a Borel subgroup in $G$, and $\mf{g},\mf{b}$ denote the Lie algebras of $G,B$ respectively, the Grothendieck-Springer resolution is the following map of Poisson varieties
\beq
\label{alg-res}
G \times_B \bgt \longrightarrow \g, \qquad (g,x) B\longmapsto gxg^{-1}.
\eeq
Indeed, the Poisson map \eqref{alg-res} admits a quantization, yielding an embedding of the enveloping algebra $U(\g)$ into the ring of global differential operators on the principal affine space $G/N$.

It was shown in \cite{EL07} that both sides of the multiplicative Grothendieck-Springer resolution
\beq
\label{mult-res}
G \times_B B \longrightarrow G, \qquad (g,b)B \longmapsto gbg^{-1}
\eeq
admit natural, nontrivial Poisson structures such that the resolution map is Poisson.  In \cite{SS15}, we showed that the resolution \eqref{mult-res} can be also quantized, this time to yield an embedding of the quantized universal enveloping algebra $U_q(\mf{g})$ into a certain ring of quantum differential operators on $G/N$.

One remarkable property of $U_q(\mf{g})$ is that it can be realized as a quotient of the Drinfeld double $D(U_q(\mf{b}))$ of a quantum Borel subalgebra $U_q(\mf{b})\subset U_q(\g)$. In this note, we observe that an analog of the quantization of the resolution~\eqref{mult-res} exists for the Drinfeld double $D(A)$ of any finite-dimensional Hopf algebra $A$. The key to the construction of this quantization is the existence of a pair $\mu_L,\mu_R \colon D(A)\rightarrow H(D(A)^{*,op})$ of commuting quantum moment maps from $D(A)$ to the Heisenberg double of its opposite Hopf dual $D(A)^{*,op}$.  In this general setting, the role of quantum differential operators on $G/N$ is played by the quantum Hamiltonian reduction of $H(D(A)^{*,op})$ by $\mu_L(A)$, and the resolution map is given by the residual quantum moment map $\mu_R \colon D(A)\rightarrow H(D(A)^{*,op})\sslash \mu_L(A)$.

Although a similar construction has appeared before in the context of the quantum Beilinson-Bernstein theorem, we believe that the following results are new. First, we show that the quantum Hamiltonian reduction $H(D(A)^{*,op})\sslash \mu_L(A)$ is isomorphic to the Heisenberg double $H(A)$. Recall, that the Heisenberg double $H(A)$ of a finite-dimensional Hopf algebra $A$ is isomorphic to the algebra of its endomorphisms $\End(A)$. Thus, the natural action of $D(A)$ on $A$ yields a homomorphism $D(A) \to H(A)$. We show that it coincides with the map $\mu_R \colon D(A) \to H(A)$. Second, we provide an explicit Faddeev-Reshetikhin-Takhtajan type presentation of the map $\mu_R$ in terms of universal $R$-matrices, which leads to a homomorphism between certain reflection equation algebras.

The article is organized as follows. In Section~\ref{poisson}, we recall the Poisson geometric constructions of~\cite{EL07} in the setting of the double $\D(G)$ of an arbitrary Poisson-Lie group $G$. It serves as a quasi-classical limit of constructions in Sections~\ref{hopf} and~\ref{quantum-res}. In Section~\ref{hopf}, we review some of the notions from the theory of Hopf algebras and their doubles that we use in the sequel. Section~\ref{quantum-res} contains the construction of the quantum resolution. Section~\ref{RTT} is devoted to the Faddeev-Reshetikhin-Takhtajan presentation of our construction. We conclude in Section~\ref{discussion} with a more detailed discussion of the quantized Grothendieck-Springer resolution.

\section*{Acknowledgements}

We express deep gratitude to Nicolai Reshetikhin, who let us use some of his unpublished results in Section~\ref{RTT}. We also wish to thank David Ben-Zvi, Arkady Berenstein, Alexandru Chirvasitu, Iordan Ganev, and David Jordan for fruitful discussions. The second author was partially supported by the NSF grant DGE-1106400 and by the RFBR grant 14-01-00547.

\section{Poisson geometry}\label{poisson}

\subsection{Preliminaries}

Recall that a Poisson-Lie group is a Lie group $G$ with a Poisson structure such that the multiplication map $G \times G \to G$ is a morphism of Poisson varieties. Let $G^*$ be the (connected, simply-connected) Poisson-Lie dual of $G$, and $\D(G)$ be the double of $G$. The Lie algebra $\dgt = Lie(\D(G))$ can be written as $\dgt = \g \oplus \g^*$. We will say, that there exist local isomorphisms
$$
\D(G) \simeq G \times G^* \simeq G^* \times G.
$$

Let us consider a pair of dual bases $(x_i)$ and $(x^i)$ of the Lie algebras $\g$ and $\g^*$ respectively. Then the element $r \in \dgt\wedge\dgt$ defined by
$$
r = \frac12\sums{i}(x_i,0) \wedge (0,x^i)
$$
is independent of the choice of bases. Let $X^R$, $X^L$ denote respectively the right- and left-invariant tensor fields on a Lie group, taking the value $X^R(e) = X^L(e) = X$ at the identity element of the group. Then the bivectors
$$
\pi_\pm = r^R \pm r^L
$$
define a pair of Poisson structures on the Lie group $\D(G)$. We abbreviate the resulting Poisson manifolds by $\D_\pm(G)$. In fact, $\D_-(G)$ is a Poisson-Lie group, while $\D_+(G)$ is only a Poisson manifold.

\begin{remark}
\label{rem-dual}
It is clear that $\D_-(G)$ is locally isomorphic to $\D_-(G)^*$ as Lie groups. However, this isomorphism does not, in general, extend to the isomorphism of Poisson-Lie groups.
\end{remark}

The action of a Poisson-Lie group $G$ on a Poisson variety $P$ is said to be Poisson, if so is the map $G \times P \to P$. Given a Poisson map $P \to G^*$, one can obtain a local Poisson action $G \times P \to P$ using the group-valued moment map. Recall, that the group-valued moment map is defined (see~\cite{Lu91}) as follows.

\begin{defn}
Let $\pi$ be the Poisson bivector field, defining the Poisson structure on the manifold $P$. A map $\mu \colon P \to G^*$ is said to be a moment map for the Poisson action $G \times P \to P$, if for every $X \in \g$ one has
$$
\mu_X = \ha{\pi, \mu^*X^R \otimes -},
$$
where $\mu_X$ is the vector field on $P$ generated by the action $\mu_{\exp(tX)}$.
\end{defn}

\begin{remark}
A moment map is Poisson, if exists.
\end{remark}

The following theorem is well-known (see e.g.~\cite{Lu91})

\begin{prop}
Let $G$ be a Poisson-Lie group, and $\D_\pm(G)$ its double with Poisson bivectors $\pi_\pm$. Then
\begin{enumerate}
\item actions of $\D_-(G)$ on $\D_+(G)$ by left and right multiplications are Poisson;
\item the moment map for the Poisson action of the subgroup $G \subset \D_-(G)$ on $\D_+(G)$ by left (resp. right) multiplication is the natural projection $\D(G) \to \D(G)/G$ (resp. $\D(G) \to G \backslash \D(G)$). Here we use the local isomorphisms $G \backslash \D(G) \simeq G^* \simeq \D(G)/G$.
\end{enumerate}
\end{prop}

\begin{remark}
\label{rem-functor}
Let $\mathcal P$ be the category of Poisson-Lie groups. Consider a Poisson-Lie group $G$ and its connected, simply-connected Poisson-Lie dual $G^*$. Then the map $G \to G^*$ defines a functor $\mathcal P \to \mathcal P^{op}$. Therefore, any Poisson-Lie subgroup $H \subset G$ induces a map $p \colon G^* \to H^*$. Now consider a Poisson action $G \times P \to P$ with the moment map $\mu_G$. It gives rise to the Poisson action $H \times P \to P$ with the moment map $\mu_H = p \circ \mu_G$.
\end{remark}

\subsection{Double of the double construction}

Now, let us start with the Poisson-Lie group $D = \D_-(G)$ and consider its double $\D(D) = \D(\D(G))$. The Lie algebra $\mathfrak D = Lie(\D(D))$ may be written as $\mathfrak{D} = \dgt \oplus \dgt = \dgt_\Delta \oplus \dgt^*$ where
$$
\dgt_\Delta = \hc{((x,\alpha),(x,\alpha)) \in \g \oplus \g^* \oplus \g \oplus \g^*}
$$
is the diagonal embedding of $\dgt$ into $\dgt \oplus \dgt$ and
$$
\dgt^* = \hc{((y,0),(0,\beta)) \in \g \oplus \g^* \oplus \g \oplus \g^*}.
$$
Using the local isomorphism $\D(D) \simeq D_\Delta \times D^*$ we may write the moment map $\nu_r$ for the right Poisson action of $D_\Delta \subset \D_-(D)$ on $\D_+(D)$ as
$$
\nu_r \colon \D_+(D) \longra D^*, \qquad (dg,d\alpha) \longmapsto \alpha^{-1}g
$$
for any triple of elements $g \in G$, $\alpha \in G^*$, $d \in \D(\D(G))$. Similarly, using the local isomorphism $\D(D) \simeq D^* \times D_\Delta$ we write the moment map $\nu_l$ for the left Poisson action of $D_\Delta \subset \D_-(D)$ on $\D_+(D)$ as
$$
\nu_l \colon \D_+(D) \longra D^*, \qquad (gd,\alpha d) \longmapsto g\alpha^{-1}.
$$

\subsection{Hamiltonian reduction}

Consider the Poisson action of the subgroup $D_\Delta \subset \D_-(D)$ on $\D_+(D)$ by left multiplications and the Poisson action of $G \subset D_\Delta \subset \D_-(D)$ on $\D_+(D)$ by right multiplications. Clearly, the two actions commute, because so do the left and right actions of $\D_-(D)$. We illustrate these actions as follows
$$
D_\Delta \curvearrowright \D_+(D) \curvearrowleft D_\Delta \supset G.
$$
By Remark~\ref{rem-functor}, the moment map $\mu_r$ for the right action of $G$ can be written as
$$
\mu_r \colon \D_+(D) \longra (G \times D) \backslash \D_+(D) \simeq G^*, \qquad (dg, d\alpha) \longmapsto \alpha.
$$
The Hamiltonian reduction of $\D_+(D)$ by the moment map $\mu_r$ becomes
$$
\mu_r^{-1}(e)/G_\Delta = \hc{(dg,d) \,|\, d\in D, g\in G}/G_\Delta.
$$
Therefore, we can identify
$$
\mu_r^{-1}(e)/G_\Delta \simeq D \times_G G,
$$
where $D \times_G G$ denotes the set of $G$-orbits through $D \times G$ under the right action
$$
(D \times G) \times G \longra D \times G, \qquad ((d,g),h) \longmapsto (dh, h^{-1}gh)
$$
with $g,h \in G$ and $d \in D$.

On the other hand, since the left and right $D_\Delta$-actions on $\D(D)$ commute, the variety $D \times_G G$ admits the residual $D_\Delta$-action by left multiplication. The corresponding moment map becomes
$$
\mu_l \colon D \times_G G \longra D, \qquad (q,g)G \longmapsto qgq^{-1},
$$
where we use that the target $D$ is locally isomorphic (as a Lie group) to the group $D^*$, see Remark~\ref{rem-dual}.

The following Proposition follows easily from considering Poisson bivectors for the Poisson varieties under consideration.
\begin{prop}
\label{dplus-iso}
Note, that there exists a local Poisson isomorphism
$$
\D_+(G) \longra D \times_G G, \qquad \alpha g \longmapsto (\alpha, g)G,
$$
where $g \in G$, $\alpha \in G^*$ and we identify $\D(G) \simeq G^* \times G$.
\end{prop}

Under this identification, the moment map $\mu_l$ becomes
\beq
\label{Poisson-final}
\D_+(G) \longra \D_-(G), \qquad \alpha g \longmapsto \alpha g \alpha^{-1}.
\eeq

\section{Reminder on Hopf algebras}\label{hopf}
To fix our notations, we will recall some standard notions from the theory of Hopf algebras.

\subsection{Basic notations.}

Let $A$ be a Hopf algebra over a field $\mathbb K$, with the quadruple $(m,\Delta,\epsilon,S)$ denoting the multiplication, comultiplication, counit, and antipode of $A$ respectively.  We say that algebras $A$ and $A^*$ form a dual pair if there exists a non-degenerate Hopf pairing $\langle -,- \rangle \colon A\otimes A^* \rightarrow \mathbb K$, that is a non-degenerate pairing satisfying
\begin{enumerate}
\item $\ha{ab,x}=\ha{a \otimes b, \Delta(x)}$
\item $\ha{a,xy}=\ha{\Delta(a), x  \otimes y}$
\item $\ha{1_A, -} = \epsilon_{A^*}$ and $\ha{-,1_{A^*}} = \epsilon_A$
\item $\ha{S(a),x} = \ha{a,S(x)}$
\end{enumerate}
for all $a,b\in A$ and $x,y\in A^*$. In fact, condition~(4) follows from the other three, see~\cite[Section 1.2.5, Proposition 9]{KS97}. We will also use the notation $A^{op}$ for the Hopf algebra $(A,m^{op},\Delta, S^{-1})$, and $A^{cop}$ for the Hopf algebra $(A,m,\Delta^{op}, S^{-1})$.

\subsection{Module algebras. }The category of modules $\Mod_A$ over a Hopf algebra $A$ has a monoidal structure determined by the coproduct $\Delta \colon A\rightarrow A\otimes A$. We say that $M$ is an $A${\em-module algebra} if it is an algebra object in the monoidal category $\Mod_A$, that is
$$
a \cdot 1_M = \epsilon(a) 1_M \qquad\text{and}\qquad a \cdot (mn) = (a_1 \cdot m) (a_2 \cdot n)
$$
for any $a \in A$ and $m,n \in M$.

A Hopf algebra $A$ can be naturally regarded as a module algebra over itself using the {\em adjoint action}
$$
\ad\colon A\otimes A \longrightarrow A, \qquad a\otimes b \longmapsto a \triangleright b:= a_1bSa_2.
$$
The dual Hopf algebra $A^*$ can also be regarded as a module algebra over $A$ using the {\em left coregular action}
$$
\coreg \colon A\otimes A^* \longrightarrow A^*, \qquad a\otimes x\longmapsto a\rhu x:= \langle a, x_2\rangle x_1.
$$
There is also a {\em right coregular action} of $A^{op}$ on $A^*$, defined by
$$
a \otimes x \longmapsto x \lhu a := \langle a, x_1\rangle x_2.
$$

\subsection{The Drinfeld double} If $A, A^*$ form a dual pair of algebras, then there exists a unique Hopf algebra $D(A)$ called the {\em Drinfeld double} of $A$, with the following properties:
\begin{enumerate}
\item as a coalgebra, $D(A) \simeq (A^*)^{cop}\otimes A$;
\item the maps $a\mapsto 1\otimes a$ and $x\mapsto x\otimes 1$ are embeddings of Hopf algebras;
\item let $(a_i)$ and $(x^i)$ be dual bases for $A$ and $A^*$ respectively. Then the canonical element
$$
R=\sum_i (1\otimes a_i)\otimes(x^i\otimes 1)\in D(A)^{\otimes 2},
$$
called the universal $R$-matrix of the Drinfeld double, satisfies
$$
R\Delta_D(d)=\Delta^{op}_D(d)R
$$
for all $d\in D(A)$.
\end{enumerate}
From the above properties one derives the following explicit formula for the multiplication in $D(A)$:
\beq
\label{drinfeld-mult}
(x\otimes a)(y\otimes b) = \langle a_1,y_3 \rangle\langle a_3,S^{-1}y_1 \rangle xy_2\otimes a_2b.
\eeq
It also follows from the definition of the double, that the $R$-matrix is invertible, with inverse
$$
R^{-1}=(S_D\otimes \mathrm{id})(R)
$$
and that the Yang-Baxter equation
$$
R_{12}R_{13}R_{23}=R_{23}R_{13}R_{12}\in D(A)^{\otimes 3}
$$
holds in the triple tensor product $D(A)^{\otimes 3}$.

\begin{prop}
If $A$ is a Hopf algebra and $D(A)$ its Drinfeld double, the following formula equips $A$ with the structure of a $D(A)$-module algebra:
\beq
\label{d-on-a-dual}
\begin{split}
(1\otimes a) \cdot b &= a_1bSa_2 \\
(x\otimes 1) \cdot b &= b\lhu S^{-1}x
\end{split}
\eeq
\end{prop}
In the action~\eqref{d-on-a-dual}, the Hopf subalgebra $A\subset D(A)$ acts adjointly on $A$, while the Hopf subalgebra $(A^*)^{cop}\subset D(A)$ acts by its right coregular action.

\subsection{The dual of the Drinfeld double.}

In addition to the Drinfeld double, we will also make use of its Hopf dual $T(A)=D(A)^*$.  As an algebra, $T(A)\simeq A^{op}\otimes A^*$.  The formula for its comultiplication can be found by dualizing~\eqref{drinfeld-mult} and reads
$$
\Delta_T(a\otimes x) = \left(a_1\otimes x^r x_1 x^t \right) \otimes \left(S^{-1}a_ta_2a_r\otimes x_2 \right) \in T(A)^{\otimes 2}.
$$
Similarly, the antipode in $T(A)$ can be written as
$$
S_T(a\otimes x) = a_rS^{-1}(a)S^{-1}(a_t)\otimes x^tS(x) x^r.
$$

\subsection{The Heisenberg double}

Given a Hopf algebra $A$ and its module algebra $M$, one defines their \emph{smash-product} $M \# A$ as an associative algebra on the vector space $M \otimes A$ with the multiplication given by
$$
(m \# x)(n \# y) = m(x_1 \cdot n)\# y_2b
$$
for any elements $x,y \in A$ and $m,n \in M$. Recall~\cite{Lu94}, that the \emph{Heisenberg double} $H(A)$ of an associative algebra $A$ is the smash product $H(A) = A\# A^*$ with respect to the coregular action of $A^*$ on $A$. Thus, the multiplication in $H(A)$ is determined by the formula
$$
(a\#x)(b\#y) = a(x_1\rhu b)\# x_2y = \ha{x_1,b_2} ab_1\# x_2y
$$
for any $a,b \in A$ and $x,y \in A^*$.
Note that one has the following inclusions of algebras
\begin{align*}
A \longra H(A), \qquad &a\mapsto a\#1, \\
A^* \longra H(A), \qquad &x\mapsto 1\#x.
\end{align*}
By construction, the Heisenberg double $H(A)$ acts on $A$ via
\beq
\label{heis-left}
(a\# x)\cdot_L(b) = a(x\rhu b) = \ha{x,b_2} ab_1
\eeq
In fact, $H(A)$ also acts on $A$ via
\beq
\label{heis-right}
(a\# x)\cdot_R(b) = (b\lhu S^{-1}x)Sa = \ha{x,Sb_1} b_2S^{-1}a
\eeq
The Heisenberg double $H(A)$ has the following well-known properties:

\begin{prop}
\cite{STS92}
The antipode $S_T$ of $T(A)$, when regarded as an operator $\iota \colon H(A)\rightarrow H(A)$ via
\beq
\label{iota}
\iota \colon H(A) \longrightarrow H(A), \qquad a\otimes x \longmapsto a_rS^{-1}(a)S^{-1}(a_t)\otimes x^tS(x) x^r,
\eeq
defines an algebra automorphism of $H(A)$.
\end{prop}
Note that the automorphism $\iota$ intertwines the two actions $\ref{heis-left},\ref{heis-right}$ of $H(A)$ on $A$.
\begin{cor}
\label{chiral-sub}
One has the following inclusions of algebras
\begin{align*}
A \longrightarrow H(A),  &\qquad a\mapsto \iota(a\#1) = a_rS^{-1}(a)S^{-1}(a_t)\otimes x^tx^r, \\
A^*\longrightarrow H(A), &\qquad x\mapsto \iota(1\#x) = a_rS^{-1}(a_t)\otimes x^tS(x) x^r.
\end{align*}
\end{cor}

Since the actions $(A\#1,\cdot_L)$, $(A\#1,\cdot_R)$ commute, we have
\begin{prop}\cite{STS92}
\label{commuting-subs}
The maps
\begin{align*}
A \otimes A \longrightarrow H(A), &\qquad a\otimes b\longmapsto (a\# 1)\iota(b\#1), \\
A^*\otimes A^* \longrightarrow H(A), &\qquad x\otimes y\longmapsto (x\# 1)\iota(1\#y)
\end{align*}
are homomorphisms of associative algebras.
\end{prop}

\subsection{Quantum Hamiltonian reduction}

Let us briefly recall the notion of quantum Hamiltonian reduction.  Suppose that $A$ is a Hopf algebra, $V$ is an associative algebra, $\mu \colon A\rightarrow V$ is a homomorphism of associative algebras, and $I$ is a 2-sided ideal in $A$ preserved by the adjoint action of $A$. Then, by the $\ad$-invariance of $I$, the action of $A$ on $V$ defined by the formula
$$
a\circ v = \mu(a_1)v\mu(Sa_2)
$$
descends to an action of $A$ on the $V$-module $V/\mu(I)$, where we abuse notation and write $\mu(I)$ for the left ideal in $V$ generated by $\mu(I)$. The \emph{quantum Hamiltonian reduction} $V\sslash \mu(A)$ of $V$ by the \emph{quantum moment map} $\mu \colon A \rightarrow V$ at the ideal $I$ is defined as the set of $A$-invariants
\begin{align*}
V\sslash \mu(A) := &\hr{ V/V\mu(I) }^A\\
= &\{ a\in V/V\mu(I) \;|\; a\circ v=\epsilon(a)v \quad\text{for all}\quad a\in A \}
\end{align*}
One checks that  $V\sslash\mu( A)$ inherits a well-defined associative algebra structure from that of $V$, such that $V\sslash \mu(A)$ is an $A$-module algebra.

\section{Construction of the quantum resolution}
\label{quantum-res}

\subsection{The double of a double}

Suppose that $A$ is a Hopf algebra, and let $D(A), T(A)$, and $H(A)$ be its Drinfeld double, dual to the Drinfeld double, and the Heisenberg double respectively. Consider the Heisenberg double
$$
H(T(A)^{op})=T(A)^{op} \# D(A)^{cop},
$$
of the algebra $T(A)^{op}$. One has an algebra embedding
$$
\mu_L \colon D(A) \longrightarrow H(T(A)^{op}), \qquad u\mapsto 1 \# u  \in H(T(A)^{op})
$$
which may be regarded as the quantum moment map for the following $D(A)$-module algebra structure on $H(T(A)^{op})$:
\beq
\label{red-action}
u\circ_L (\phi\# v)= (u_3\rhu \phi) \# u_2 v S_{D(A)}^{-1}u_1.
\eeq
As in Corollary \ref{chiral-sub}, there exists another algebra embedding defined by
\beq
\label{mu-prime}
\mu_R \colon D(A) \longrightarrow H(T(A)^{op}), \qquad u\mapsto \iota^{-1}(1\#u).
\eeq
It generates the following $D(A)$-module algebra structure on $H(T(A)^{op})$:
\beq
\label{resid-action}
u \circ_R (\phi\# v) = (\phi\lhu S_{D(A)}^{-1}u)\# v
\eeq

By Proposition~\ref{commuting-subs}, the subalgebras $\mu_L(D(A))$ and $\mu_R(D(A))$ commute with each other in $H(T(A)^{op})$. This forces the actions~\eqref{red-action} and~\eqref{resid-action} to commute as well.

\subsection{Dual pairs of quantum moment maps.}

We shall now restrict the action~\eqref{red-action} to the Hopf subalgebra $A\subset D(A)$, and consider the quantum Hamiltonian reduction of $H(T(A)^{op})$ at the augmentation ideal $I_A = \ker(\epsilon_A)$ of $A$.  We denote the algebra obtained as a result of the quantum Hamiltonian reduction by $H(T(A)^{op})\sslash \mu_L(A)$.

We also have the moment map $\mu_R \colon D(A)\rightarrow H(T(A)^{op})$ given in~\eqref{mu-prime}, and the action~\eqref{resid-action} of $D(A)$ on $H(T(A)^{op})$ that it defines.

\begin{prop}
The action (\ref{resid-action}) of $D(A)$ on $H(T(A)^{op})$ descends to a well-defined action
\beq
\label{one-more-action}
D(A) \times H(T(A)^{op})\sslash \mu_L(A) \longra H(T(A)^{op})\sslash \mu_L(A)
\eeq
In turn, the map $\mu_R$ descends to a well-defined homomorphism of $D(A)$-module algebras
$$
\mu_R \colon D(A) \longrightarrow  H(T(A)^{op})\sslash \mu_L(A)
$$
which is a moment map for the action~\eqref{one-more-action}.
\end{prop}
\begin{proof}
The Proposition is a simple consequence of the fact that the subalgebras $\mu_L(D(A))$ and $\mu_R(D(A))$ commute with one another. Indeed, this commutativity implies that for all $a\in A$, $u\in D(A)$, one has
\begin{align*}
a\circ_L \hr{ \mu_R(u) + \mu_L(I_A) } &= a_1\mu_R(u)Sa_2 +\mu_L(I_A) \\
&= \mu_R(u)a_1Sa_2 +\mu_L(I_A) \\
&= \epsilon(a)\hr{ \mu_R(u) + \mu_L(I_A) }
\end{align*}
which shows that
$$
\mu_R(u)+\mu_L(I_A) \in \br{H(T(A)^{op})/\mu_L(I_A)}^A =: H(T(A)^{op}) \sslash \mu_L(A).
$$
It follows from the definition of the algebra structure of the quantum Hamiltonian reduction $H(T(A)^{op})\sslash \mu_L(A)$ that  $\mu_R \colon D(A)\rightarrow  H(T(A)^{op})\sslash \mu_L(A)$ is a homomorphism of algebras. Regarding this homomorphism as a quantum moment map, we obtain an action of $D(A)$ on $H(T(A)^{op})\sslash \mu_L(A)$ which by construction descends from~\eqref{resid-action}, and such that $\mu_R \colon D(A) \rightarrow H(T(A)^{op})$ is a morphism of $D(A)$-module algebras.
\end{proof}

\subsection{$H(A)$ from quantum Hamiltonian reduction}

We now examine the algebra structure of the Hamiltonian reduction $H(T(A)^{op})\sslash \mu_L(A)$ in more detail.
\begin{prop}
There is an isomorphism of algebras
\begin{align}
\label{heis-iso}
\varphi:H(T(A)^{op})\sslash \mu_L(A)\longrightarrow H(A)
\end{align}
\end{prop}
\begin{proof}
Let us begin by making explicit the structure of the Hamiltonian reduction $H(T(A)^{op})\sslash \mu_L(A)$.  Firstly, note that we can identify the quotient $H(T(A)^{op})/I_A$ with the vector space $T(A)^{op}\otimes A^*$. It is easy to check that induced action of $A$ on $T(A)^{op}\otimes A^*$ is then given by
$$
a\circ_L \br{(b\otimes y)\otimes x} = \hr{ b\otimes a_2\rhu y } \otimes \ad^*_{a_1}(x)
$$
where
$$
\ad^*_a(x) = \langle a_1, x_3\rangle\langle S^{-1}a_2, x_1\rangle x_2.
$$
Hence the algebra $H(T(A)^{op})\sslash \mu_L(A)$ of $A$-invariants in $H(T(A)^{op})/I_A$ may be identified with  $H(A)=A\# A^*$, as a vector space, under the map
\beq
\label{inv-iso}
\varphi \colon H(A) \longrightarrow H(T(A)^{op})\sslash \mu_L(A), \qquad a\# x \longmapsto (a\otimes x_1Sx_3) \otimes x_2.
\eeq
Finally, we claim that the map (\ref{inv-iso}) is in fact an isomorphism of algebras.  Indeed, in $ H(T(A)^{op})\sslash \mu_L(A)$, one computes
\begin{align*}
\varphi(a\# x)\varphi(b\# y)&=\hr{(a\otimes x_1Sx_3) \otimes x_2}\hr{(b\otimes y_1Sy_3) \otimes y_2}\\
&=\langle x_2, S^{-1}a_tb_2a_r\rangle(ab_1\otimes x^ry_1Sy_3x^tx_1Sx_3)\otimes x_3y_2\\
&=\langle x_3, b_2\rangle (ab_1\otimes x_4y_1Sy_3S^{-1}x_2x_1Sx_6)\otimes x_5y_2\\
&=\langle x_1, b_2\rangle (ab_1\otimes x_2y_1Sy_3Sx_4)\otimes x_3y_2\\
&=\varphi(\langle x_1,b_2\rangle ab_1\otimes x_2y)\\
&=\varphi\hr{(a\# x)(b\# y)}
\end{align*}
which completes the proof.
\end{proof}

\begin{cor}
Under the isomorphism $\varphi$ defined in~\eqref{heis-iso}, the moment map $\mu_R \colon D(A)\rightarrow  H(T(A)^{op})\sslash \mu_L(A)\simeq H(A)$ takes the form
\beq
\label{hom}
\mu_R \colon D(A) \longrightarrow H(A), \qquad by \longmapsto b_1 a_r Sb_2 a_t \# S^{-1}x^t S^{-1}y x^r.
\eeq
\end{cor}
Using the homomorphism $\mu_R$, one can pull back the defining representation $\eqref{heis-left}$ of $H(A)$ on $A$ to obtain a representation of $D(A)$.  A straightforward computation establishes
\begin{prop}
\label{pullback}
The pullback under $\mu_R$ of the action \eqref{heis-left} coincides with the representation \eqref{d-on-a-dual} of $D(A)$ on $A$.
\end{prop}
\begin{remark} In \cite{Lu96}, the formula \eqref{hom} is derived from the action \eqref{d-on-a-dual} of $D(A)$, together with the fact that $H(A)\simeq \mathrm{End}(A)$ as algebras.
\end{remark}

\section{$R$-matrix formalism}
\label{RTT}

In this section we rewrite the homomorphism~\eqref{hom} in terms of canonical elements of the algebras $D(A)$ and $T(A)$. As before, let
$$
R=R_{12}=\sum_{i} a_i \otimes x^i \in D(A) \otimes D(A)
$$
be the universal $R$-matrix of $D(A)$. In what follows we make use of elements
$$
R_{21}=\sum_{i} x^i \otimes a_i
$$
and
$$
\mathcal{L} = R_{21}R_{12} \in D(A) \otimes D(A).
$$
Recall \cite{STS92} that the element $\mathcal{L}$ satisfies the {\em reflection equation}
\beq
\label{ref-eq}
\mathcal{L}_1R_{12}\mathcal{L}_2R_{21} = R_{12}\mathcal{L}_2 R_{21}\mathcal{L}_1 \in D(A)^{\otimes 3}
\eeq
where $\mathcal{L}_1= R_{31}R_{13}, \ \mathcal{L}_2=R_{32}R_{23}$.
Let us also introduce canonical elements $\Theta, \Omega \in D(A) \otimes H(A)$ defined by
$$
\Theta = \sum_i a_i\otimes x_i \qquad\text{and}\qquad \Omega = \sum_i x^i\otimes a_i.
$$
These elements satisfy the relations
\beq
\begin{split}\label{heis-rels}
&R_{12}\Theta_1\Theta_2 = \Theta_2\Theta_1R_{12}\\
&R_{12}\Omega_1\Omega_2 = \Omega_2\Omega_1R_{12} \\
&R_{12}\Theta_1\Omega_2^{-1} = \Omega_2^{-1}\Theta_1
\end{split}
\eeq
If $\iota$ is the automorphism of $H(A)$ defined by~\eqref{iota}, we write
$$
\widetilde{\Theta} = \hr{\id\otimes\iota}(\Theta) \qquad\text{and}\qquad \widetilde{\Omega} = \hr{\id\otimes\iota}(\Omega).
$$

The following proposition is straightforward.

\begin{prop}
Let $\mu_R \colon D(A) \to H(A)$ be the homomorphism defined by~\eqref{hom}. Then one has
\begin{align*}
\hr{\id\otimes\mu_R}(R_{12}) &= \widetilde{\Theta}, \\
\hr{\id\otimes\mu_R}(R_{21}) &= \Omega\widetilde{\Omega},
\end{align*}
and hence
$$
\hr{\id\otimes\mu_R}(\mathcal{L}) = \Omega\widetilde{\Omega}\widetilde{\Theta}.
$$
\end{prop}

Recall~\cite[Section 8.1.3, Proposition 5]{KS97} that the element $u \in D(A)$ defined by
$$
u=Sa_iSx^i\in D(A)
$$
satisfies
$$
udu^{-1} = S_D^{2}(d) \qquad\text{for all}\qquad d\in D(A).
$$
\begin{prop}
The following identity holds in $D(A) \otimes H(A)$
$$
\Theta^{-1}\Omega^{-1}=u_1\widetilde{\Omega}\widetilde{\Theta},
$$
where $u_1=u\otimes 1\in D(A) \otimes H(A)$.
\end{prop}

\begin{proof}
We have
\begin{align*}
\Theta^{-1}\Omega^{-1}& = \sum Sa_iSx^j \otimes x^ia_j \\
&= \sum S(a_ka_t)S(x^rx^k) \otimes (a_r\# x^t) \\
&= \sum a_t u x^r \otimes (Sa_r\# Sx^t) \\
&= u_1\sum a_t x^r \otimes (Sa_r\# S^{-1}x^t).
\end{align*}
Using the formula
$$
ax = \langle a_{(1)},x_{(3)}\rangle \langle S^{-1}a_{(3)},x_{(1)}\rangle x_{(2)}\otimes a_{(2)}
$$
for the multiplication in the Drinfeld double $D(A)$, we arrive at
\begin{align*}
\Theta^{-1}\Omega^{-1} &= u_1\sum a_t x^r \otimes (Sa_r\# S^{-1}x^t) \\
&= u_1\sum x^p a_q\otimes \hr{ a_\alpha Sa_q S^{-1}a_\beta\# x^\beta S^{-1}x^px^\alpha} \\
&=u_1\widetilde{\Omega}\widetilde{\Theta}
\end{align*}
which completes the proof.
\end{proof}

\begin{cor}
One has
\beq
\label{RTT-final}
\hr{\id\otimes\mu_R}(\mathcal{L}) = \Omega u_1^{-1}\Theta\Omega^{-1}.
\eeq
\end{cor}

\begin{remark}
Since the first tensor factor in $\mathcal L$ runs over the basis of $D(A)$, the homomorphism~$\mu_R$ is completely defined by the formula~\eqref{RTT-final}. The latter can be thought of as a quantization of the map~\eqref{Poisson-final}, where $u_1$ is a quantum correction, invisible on the level of Poisson geometry.
\end{remark}

\begin{cor}
\label{ref-sol}
The element
\beq
\label{L-hat}
\widehat{\mathcal{L}} = \Omega u_1^{-1}\Theta\Omega^{-1} \in D(A)\otimes H(A)
\eeq
\end{cor}
provides a solution to the reflection equation \eqref{ref-eq}.

\begin{remark}
In fact, one can check using the relations \eqref{heis-rels} that the element
$$
\widehat{\mathcal{L}}'=\Omega \Theta\Omega^{-1} \in D(A)\otimes H(A)
$$
obtained from~\eqref{L-hat} by omitting $u_1^{-1}$, also satisfies the reflection equation \eqref{ref-eq}.  In general, however, the linear map $D(A)\rightarrow H(A)$ defined by $\mathcal{L}\mapsto \widehat{\mathcal{L}}'$ will fail to be a homomorphism of algebras.  On the other hand, suppose that $R \in \mathrm{End}(V\otimes V)$ is a scalar solution of the Yang-Baxter equation. Then, following Faddeev-Reshetikhin-Takhtajan, one can define a {\em reflection equation algebra} $\mathcal{A}$ as the algebra generated by entries of $\mathcal L\in \mathcal{A}\otimes\mathrm{End}(V)$, subject to the defining relations \eqref{ref-eq}. Similarly, one can define an algebra $\mathcal{H}$ generated by entries of the elements $\Theta,\Omega\in \mathcal{H}\otimes\mathrm{End}(V)$ subject to the relations \eqref{heis-rels}. Then we get a well-defined homomorphism of algebras
$$
\mathcal A \longrightarrow \mathcal H, \qquad \mathcal{L} \longmapsto \Omega \Theta\Omega^{-1}.
$$
\end{remark}

\section{Quantized Grothendieck-Springer resolution}
\label{discussion}

Suppose that $\g$ is a complex simple Lie algebra, and denote by $U_\hbar(\g)$ the quantized universal enveloping algebra of $\g$, see \cite{Dri86,CP94}.  Recall that $U_\hbar(\mf{g})$ may be regarded as the quantized algebra of functions on a formal neighborhood of the identity element $e \in G^*$, where $G$ is a simple Lie group endowed with its standard Poisson structure. Let us apply our constructions to the case $A=U_\hbar(\mf{b})$, where $U_\hbar(\mf{b})$ is the quantum Borel subalgebra in $U_\hbar(\g)$.  Then there is an isomorphism of algebras $D(A)\simeq U_\hbar(\g)\otimes U_\hbar(\mf{h})$, where $\mf{h}\subset\mf{g}$ is the Cartan subalgebra of $\g$, see \cite{Dri86}.  The restriction of the homomorphism \eqref{hom} to  $U_\hbar(\g)\subset D(A)$ defines a map of algebras $\Phi\colon U_\hbar(\g)\rightarrow H(A)$.  In \cite{SS15}, it was shown (in the setting of the rational form $U_q(\g)$) that $\Phi$ is injective, and that its image is contained in a certain subalgebra $H(A)^{\mf{h}}$ of $U_\hbar(\mf{h})$-invariants.

In the above setup, the semiclassical limit of the map $\Phi$ is closely related to the well-known Grothendieck-Springer resolution
$$
G\times_B B\longrightarrow G, \qquad (g,b)B\longmapsto gbg^{-1}
$$
where $G$ is a complex simple Lie group, and $B \subset G$ is a Borel subgroup. More precisely, the algebra $H(A)^{\mf{h}}$ can be regarded as the quantized algebra of functions on a formal neighborhood of $(e,e)B \in G\times_B B$. The Poisson geometric structure is exactly the one described in~\cite{EL07}.

\bibliographystyle{alpha}

\end{document}